\documentclass{amsart}

\usepackage{latexsym}
\usepackage{amssymb, latexsym}
\usepackage{amsmath}
\usepackage{mathrsfs}
\usepackage{amsxtra}
\usepackage[balance]{diagrams}
\usepackage{amsthm}
\newtheorem{theorem}{Theorem}[section]
\newtheorem{corollary}[theorem]{Corollary}

\newtheorem{lemma}[theorem]{Lemma}

\newtheorem{proposition}[theorem]{Proposition}

\theoremstyle{definition}
\newtheorem{definition}[theorem]{Definition}
\newtheorem{question}[theorem]{Question}
\newtheorem{remark}[theorem]{Remark}

\newcommand{\p}{\mathbb{P}}
\newcommand{\q}{\mathbb{Q}}

\newcommand{\I}{\mathscr{I}}

\newcommand{\la}{\langle}
\newcommand{\ra}{\rangle}

\newcommand{\her}[1]{H_{{#1}^+}}

\newcommand{\lset}[1]{\langle L_{#1}[A],A\rangle}
\newcommand{\lsetex}[1]{\langle L_{#1}[A],A,\xi\rangle}

\newcommand{\Power}{\mathcal P}

\newcommand{\pl}{\emph{(}}
\newcommand{\pr}{\emph{)}}

\newcommand{\restrict}{\upharpoonright}
\newcommand{\forces}{\Vdash}

\newarrow{Corresponds}<--->
\newarrow{Dashto}{}{dash}{}{dash}>
\newarrow {Dotsto} .....

\newcommand{\tmop}[1]{\ensuremath{\operatorname{#1}}}
\newcommand{\tmtextit}[1]{{\itshape{#1}}}
\begin{document}
\author{Victoria Gitman}
\author{P.D. Welch}
\today
\address{New York City College of Technology (CUNY), 300 Jay Street,
Brooklyn, NY 11201 USA} \email{vgitman@nylogic.org}
\address{University of Bristol, School of Mathematics
University of Bristol Clifton, Bristol, BS8 1TW United
Kingdom}\email{P.Welch@bristol.ac.uk}
\thanks{The research of the first author has been supported by grants from the
CUNY Research Foundation.}
\title{Ramsey-like Cardinals II}
\maketitle
\begin{abstract}
This paper continues the study of the \emph{Ramsey-like} large
cardinals introduced in \cite{gitman:ramsey} and
\cite{welch:ramsey}. Ramsey-like cardinals are defined by
generalizing the characterization of Ramsey cardinals via the
existence of elementary embeddings. Ultrafilters derived from such
embeddings are fully iterable and so it is natural to ask about
large cardinal notions asserting the existence of ultrafilters
allowing only $\alpha$-many iterations for some countable ordinal
$\alpha$. Here we study such $\alpha$-\emph{iterable} cardinals. We
show that the $\alpha$-iterable cardinals form a strict hierarchy
for $\alpha\leq\omega_1$, that they are downward absolute to $L$ for
$\alpha<\omega_1^L$, and that the consistency strength of
Schindler's remarkable cardinals is strictly between 1-iterable and
2-iterable cardinals.

We show that the strongly Ramsey and super Ramsey cardinals from
\cite{gitman:ramsey} are downward absolute to the core model $K$.
Finally, we use a forcing argument from a strongly Ramsey cardinal
to separate the notions of Ramsey and \emph{virtually Ramsey}
cardinals. These were introduced in
\cite{welch:ramsey} as an upper bound on the consistency strength of
the Intermediate Chang's Conjecture.

\end{abstract}
\section{Introduction}
The definitions of measurable cardinals and stronger large cardinal
notions follow the template of asserting the existence of elementary
embeddings $j:V\to M$ from the universe of sets to a transitive
subclass with that cardinal as the critical point. Many large
cardinal notions below a measurable cardinal can be characterized by
the existence of elementary embeddings as well. The
characterizations of these smaller large cardinals $\kappa$ follow
the template of asserting the existence of elementary embeddings
\hbox{$j:M\to N$} with critical point $\kappa$ from a weak
$\kappa$-model or $\kappa$-model $M$ of set theory to a transitive
set.\footnote{It will be assumed throughout the paper that, unless
stated otherwise, all embeddings are elementary and between
transitive structures.} A \emph{weak} $\kappa$-\emph{model} $M$ of
set theory is a transitive set of size $\kappa$ satisfying ${\rm
ZFC}^-$ (ZFC without the \emph{Powerset Axiom}) and having
$\kappa\in M$. If a weak $\kappa$-model $M$ is additionally closed
under $<\kappa$-sequences, that is $M^{<\kappa}\subseteq M$, it is
called a $\kappa$-\emph{model} of set theory. Having embeddings on
$\kappa$-models is particularly important for forcing
indestructibility arguments, where the techniques rely on
$<\kappa$-closure. The weakly compact cardinal is one example of a
smaller large cardinal that is characterized by the existence of
elementary embeddings. A cardinal $\kappa$ is \emph{weakly compact}
if $\kappa^{<\kappa}=\kappa$ and every $A\subseteq\kappa$ is
contained in a weak $\kappa$-model $M$ for which there exists an
elementary embedding $j:M\to N$ with critical point $\kappa$.
Another example is the strongly unfoldable cardinal. A cardinal
$\kappa$ is \emph{strongly unfoldable} if for every ordinal
$\alpha$, every $A\subseteq\kappa$ is contained a weak
$\kappa$-model $M$ for which there exists an elementary embedding
$j:M\to N$ with critical point $\kappa$, $\alpha<j(\kappa)$, and
$V_\alpha\subseteq N$.

An embedding $j:V\to M$ with critical point $\kappa$ can be used to
construct a $\kappa$-complete ultrafilter on $\kappa$. These
measures are \emph{fully iterable}; they allow iterating the
ultrapower construction through all the ordinals. The iteration
proceeds by taking ultrapowers by the image of the original
ultrafilter at successor ordinal stages and direct limits at limit
ordinal stages to obtain a directed system of elementary embeddings
of \emph{well-founded} models of length $Ord$. Returning to smaller
large cardinals, an embedding $j:M\to N$ with critical point
$\kappa$ and $M$ a model of $\rm{ZFC}^-$ can be used to construct an
ultrafilter on $\Power(\kappa)^M$ that is $\kappa$-complete
\emph{from the perspective} of $M$. These small measures are called
$M$-\emph{ultrafilters} because all their measure-like properties
hold only from the perspective of $M$. It is natural to ask what
kind of iterations can be obtained from $M$-ultrafilters. Here,
there are immediate technical difficulties arising from the fact
that an $M$-ultrafilter is, in most interesting cases, external to
$M$. In order to start defining the iteration, the $M$-ultrafilter
needs to have the additional property of being \emph{weakly
amenable}. The existence of weakly amenable $M$-ultrafilters on
$\kappa$ with well-founded ultrapowers is equivalent to the
existence of embeddings $j:M\to N$ with critical point $\kappa$
where $M$ and $N$ have the \emph{same} subsets of $\kappa$. We will
call such embeddings $\kappa$-\emph{powerset preserving}.

Gitman observed in \cite{gitman:ramsey} that weakly compact
cardinals are not strong enough to imply the existence of
$\kappa$-powerset preserving embeddings. She called a cardinal
$\kappa$ \emph{weakly Ramsey} if every $A\subseteq\kappa$ is
contained in a weak $\kappa$-model for which there exists a
$\kappa$-powerset preserving elementary embedding $j:M\to N$. In
terms of consistency strength weakly Ramsey cardinals are above
completely ineffable cardinals and therefore much stronger than
weakly compact cardinals. We associate iterating $M$-ultrafilters
with Ramsey cardinals because Ramsey cardinals imply the existence
of fully iterable $M$-ultrafilters. Mitchell showed in
\cite{mitchell:ramsey} that $\kappa$ is Ramsey if and only if every
$A\subseteq\kappa$ is contained in a weak $\kappa$-model $M$ for
which there exists a weakly amenable countably complete\footnote{An
$M$-ultrafilter is \emph{countably complete} if every countable
collection of sets in the ultrafilter has a nonempty intersection
(see Section \ref{sec:preliminary}).} $M$-ultrafilter on $\kappa$.
Kunen showed in \cite{kunen:ultrapowers} that countable completeness
is a sufficient condition for an $M$-ultrafilter to be fully
iterable, that is, for every stage of the iteration to produce a
well-founded model. The \hbox{$\alpha$-iterable} cardinals were
introduced in \cite{gitman:ramsey} to fill the gap between weakly
Ramsey cardinals that merely assert the existence of
$M$-ultrafilters with the potential to be iterated and Ramsey
cardinals that assert the existence of fully iterable
$M$-ultrafilters. A cardinal $\kappa$ is $\alpha$-\emph{iterable} if
every subset of $\kappa$ is contained in a weak $\kappa$-model $M$
for which there exists an $M$-ultrafilter on $\kappa$ allowing an
iteration of length $\alpha$. By a well-known result of Gaifman
\cite{gaifman:ultrapowers}, an ultrafilter that allows an iteration
of length $\omega_1$ is fully iterable. So it only makes sense to
study the $\alpha$-iterable cardinals for $\alpha\leq\omega_1$.

Welch and Sharpe showed in \cite{welch:ramsey} that
$\omega_1$-iterable cardinals are strictly weaker than
$\omega_1$-Erd\H{o}s cardinals. In Section \ref{sec:iterable}, we
show that for $\alpha<\omega_1^L$, the $\alpha$-iterable cardinals
are downward absolute to $L$. In Section \ref{sec:hierarchy}, we
show that for $\alpha\leq\omega_1$,  the $\alpha$-iterable cardinals
form a hierarchy of strength. Also, in Section \ref{sec:hierarchy},
we establish a relationship between $\alpha$-iterable cardinals and
$\alpha$-Erd\H{o}s cardinals, and provide an improved upper bound on
the consistency strength of Schindler's remarkable cardinals by
placing it strictly between 1-iterable cardinals and 2-iterable
cardinals. Finally we answer a question of Gitman about whether
1-iterable cardinals imply existence of embeddings on weak
$\kappa$-models of ZFC.

Gitman also introduced \emph{strongly Ramsey} cardinals and
\emph{super Ramsey} cardinals by requiring the existence of
$\kappa$-powerset preserving embeddings on $\kappa$-models instead
of weak $\kappa$-models. The strongly Ramsey and super Ramsey
cardinals fit in between Ramsey cardinals and measurable cardinals
in strength. In Section \ref{sec:K}, we show that these two large
cardinal notions are downward absolute to the core model $K$. In
Section \ref{sec:vram}, we use a forcing argument starting from a
strongly Ramsey cardinal to separate the notions of virtually Ramsey
and Ramsey cardinals. \emph{Virtually Ramsey} cardinals were
introduced by Welch and Sharpe in \cite{welch:ramsey} as an upper
bound on the consistency of the Intermediate Chang's Conjecture.
\section{Preliminaries}\label{sec:preliminary}
In this section, we review facts about $M$-ultrafilters and formally
define the $\alpha$-iterable cardinals. We begin by giving a precise
definition of an $M$-ultrafilter.
\begin{definition}
Suppose $M$ is a transitive model of $\rm{ZFC}^-$ and $\kappa$ is a
cardinal in $M$. A set $U\subseteq \Power(\kappa)^M$ is an
$M$-\emph{ultrafilter} if $\la M,\in,U\ra\models ``U$ is a
$\kappa$-complete normal ultrafilter".
\end{definition}
Recall that an ultrafilter is $\kappa$-\emph{complete} if the
intersection of any $<\kappa$-sized collection of sets in the
ultrafilter is itself an element of the ultrafilter. An ultrafilter
is \emph{normal} if every function regressive on a set in the
ultrafilter is constant on a set in the ultrafilter. By definition,
$M$-ultrafilters are $\kappa$-complete and normal only from the
\emph{point of view} of $M$, that is, the collection of sets being
intersected or the regressive function has to be an element of $M$.
We will say that an $M$-ultrafilter is \emph{countably complete} if
every countable collection of sets in the ultrafilter has a nonempty
intersection. Obviously, any $M$-ultrafilter is, by definition,
countably complete from the point of view of $M$, but countable
completeness requests the property to hold of \emph{all} sequences,
not just those in $M$.\footnote{It is more standard for countable
completeness to mean $\omega_1$-completeness which requires the
intersection to be an element of the ultrafilter. However, the
weaker notion we use here is better suited to $M$-ultrafilters
because the countable collection itself can be external to $M$, and
so there is no reason to suppose the intersection to be an element
of $M$.} Unless $M$ satisfies some extra condition, such as being
closed under countable sequences, an $M$-ultrafilter need not be
countably complete. In this article we shall consider the usual
ultrapower of a structure $M$ taken using only functions in $M$. We
are thus not using fine-structural ultrapowers in our arguments. An
ultrapower by an $M$-ultrafilter is not necessarily well-founded. An
$M$-ultrafilter with a well-founded ultrapower may be obtained from
an elementary embedding $j:M\to N$.
\begin{proposition}\label{prop:genult}
Suppose $M$ is a weak $\kappa$-model and $j:M\to N$ is an elementary
embedding with critical point $\kappa$, then
$U=\{A\in\Power(\kappa)^M\mid \kappa\in j(A)\}$ is an
$M$-ultrafilter on $\kappa$ with a well-founded ultrapower.
\end{proposition}
In this case, we say that $U$ is \emph{generated by $\kappa$ via
  j}. The well-foundedness of the ultrapower follows since it embeds
into $N$.

To define $\alpha$-iterable cardinals, we will need the
corresponding key notion of $\alpha$-\emph{good} $M$-ultrafilters.
\begin{definition}
Suppose $M$ is a weak $\kappa$-model. An $M$-ultrafilter $U$ on
$\kappa$ is $0$-\emph{good} if the ultrapower of $M$ by $U$ is
well-founded.
\end{definition}

To begin discussing the iterability of $M$-ultrafilters, we need the
following key definitions.
\begin{definition}\label{d:wa}
Suppose $M$ is a weak $\kappa$-model. An $M$-ultrafilter $U$ on
$\kappa$ is \emph{weakly amenable} if for every $A\in M$ of size
$\kappa$ in $M$, the intersection $U\cap A$ is an element of $M$.
\end{definition}
\begin{definition}
Suppose $M$ is a model of $\rm{ZFC}^-$. An elementary embedding
\hbox{$j:M\to N$} with critical point $\kappa$ is
$\kappa$-\emph{powerset preserving} if $M$ and $N$ have the same
subsets of $\kappa$.
\end{definition}
It turns out that the existence of weakly amenable 0-good $M$-ultrafilters
on $\kappa$ is equivalent to the
existence of $\kappa$-powerset preserving embeddings.
\begin{proposition}\label{p:wa}Suppose $M$ is a transitive model of $\rm{ZFC}^-$.
\begin{itemize}
\item[(1)] If $j:M\to N$ is the ultrapower by a weakly amenable
$M$-ultrafilter on $\kappa$, then $j$ is $\kappa$-powerset
preserving.
\item[(2)] If $j:M\to N$ is a $\kappa$-powerset preserving embedding,
then the $M$-ultrafilter $U=\{A\in\Power(\kappa)^M\mid \kappa\in
j(A)\}$ is weakly amenable. \end{itemize}
\end{proposition}
\begin{definition}
Suppose $M$ is a weak $\kappa$-model. An $M$-ultrafilter on $\kappa$
is 1-\emph{good} if it is 0-good and weakly amenable.
\end{definition}

\begin{lemma}\label{le:iterate}
Suppose $M$ is a weak $\kappa$-model, $U$ is a $1$-good
$M$-ultrafilter on $\kappa$, and $j:M\to N$ is the ultrapower by
$U$. Define
\begin{displaymath}
j(U)=\{A\in\Power( j(\kappa))^N\mid A=[f]\text{ and
}\{\alpha\in\kappa\mid f(\alpha)\in U\}\in U\}.
\end{displaymath}
Then $j(U)$ is a weakly amenable $N$-ultrafilter on $j(\kappa)$
containing $j''U$ as a subset.
\end{lemma}
See \cite{kanamori:higher} for details on the above facts. Lemma
\ref{le:iterate} is essentially saying that the weak amenability of
$U$ implies a partial  \L o\'{s} Theorem for the ultrapower of
\hbox{$\la M,\in, U\ra$} by $U$ resulting in $j(U)$, the predicate
corresponding to $U$ in the ultrapower, having the requisite
properties. The resulting ultrapower is fully elementary in the
language without the predicate for $U$ and $\Sigma_0$-elementary
with the predicate. This suffices since the main purpose in taking
the ultrapower in the extended language is to obtain the next
ultrafilter in the iteration. Weak amenability serves as the basis
of any fine structural analysis of measures and extenders
\cite{zeman:inner}. As we shall see later, it is not necessarily the
case that the ultrapower by $j(U)$ is well-founded.

Suppose $M$ is a weak $\kappa$-model and $U_0$ is a 1-good
$M$-ultrafilter on $\kappa$. Let $j(U_0)=U_1$ be the weakly amenable
ultrafilter obtained as above for the ultrapower of $M$ by $U$. If
the ultrapower by $U_1$ happens to be well-founded, we will say that
$U_0$ is 2-\emph{good}. In this way, we can continue iterating the
ultrapower construction so long as the ultrapowers are well-founded.
For $\xi\leq\omega$, we will say that $U$ is $\xi$-\emph{good} if
the first $\xi$-many ultrapowers are well-founded. Suppose next that
the first $\omega$-many ultrapowers are well-founded. We can form
their direct limit and ask if that is well-founded as well. If the
direct limit of the first $\omega$-many iterates turns out to be
well-founded, we will say that $U$ is $\omega+1$-\emph{good}.
Continuing the pattern, we make the following definition.
\begin{definition}
Suppose $M$ is a weak $\kappa$-model and $\alpha$ is an ordinal.  An
$M$-ultrafilter on $\kappa$ is $\alpha$-\emph{good}, if we can
iterate the ultrapower construction for $\alpha$-many steps.
\end{definition}
Gaifman showed in \cite{gaifman:ultrapowers} that to be able to
iterate the ultrapower construction through all the ordinals it
suffices to know that we can iterate through all the countable
ordinals.
\begin{theorem}\label{th:gaifman}
Suppose $M$ is a weak $\kappa$-model. An $\omega_1$-good
$M$-ultrafilter is $\alpha$-good for every ordinal $\alpha$.
\end{theorem}
Thus, the study of $\alpha$-good ultrafilters only makes sense for
$\alpha\leq\omega_1$.
\begin{definition}
For $\alpha\leq\omega_1$, a cardinal $\kappa$ is
$\alpha$-\emph{iterable} if every $A\subseteq\kappa$ is contained in
a weak $\kappa$-model $M$ for which there exists an $\alpha$-good
$M$-ultrafilter on $\kappa$.
\end{definition}
A few easy observations about the definition are in order.
\begin{remark}\label{rem:iterablecards}$\,$\\
\begin{itemize}
\item[(1)] If $\kappa^{<\kappa}=\kappa$, then $\kappa$ is 0-iterable if and
only if $\kappa$ is weakly compact. Without the extra assumption
$\kappa^{<\kappa}=\kappa$, being 0-iterable is not necessarily a
large cardinal notion. Hamkins showed in \cite{hamkins:book} that it
is consistent for $2^\omega$ to be 0-iterable.
\item[(2)] Weakly Ramsey cardinals are exactly the
$1$-iterable cardinals. Unlike\break 0-iterability, 1-iterability
implies inaccessibility and hence weak compactness (see
\cite{gitman:ramsey} for the strength of 1-iterable cardinals).
\item[(3)] By our previous comments, Ramsey cardinals
are $\omega_1$-iterable.
\item[(4)] $\omega_1$-iterable
cardinals are \emph{strongly unfoldable} in $L$ (see
\cite{villaveces:unfoldable}).
\end{itemize}
\end{remark}
\section{$\alpha$-iterable cardinals in $L$}\label{sec:iterable}
In this section, we show that for $\alpha<\omega_1^L$, the
$\alpha$-iterable cardinals are downward absolute to $L$. This
result is optimal since $\omega_1$-iterable cardinals cannot exist
in $L$.

Many of our arguments below will use the following two simple facts
about weak $\kappa$-models.
\begin{remark}\label{rem:model}$\,$\\
\begin{itemize}
\item[(1)] If $M$ is a weak $\kappa$-model of height $\alpha$, then
$L^M=L_\alpha$. Note that $M\cap L$ can be a proper superset of
$L_\alpha$. That is, $M$ might contain constructible elements that
it does not realize are constructible.
\item[(2)] If $M$ is a weak $\kappa$-model, $j:M\to N$ is an elementary embedding with critical point
$\kappa$,
and $X$ has size $\kappa$ in $M$, then $j\restrict X$ is an element
of $N$. This follows since $j\restrict X$ is definable from an
enumeration $f$ of $X$ in $M$ together with $j(f)$, both of which
are elements of $N$.
\end{itemize}
\end{remark}
Next, we give an argument why $\omega_1$-iterable cardinals cannot
exist in $L$.
\begin{proposition}
If there is an $\omega_1$-iterable cardinal, then $0^\#$ exists.
\end{proposition}
\begin{proof}
Suppose $\kappa$ is an $\omega_1$-iterable cardinal. Fix a weak
$\kappa$-model $M$ and an \hbox{$\omega_1$-good} $M$-ultrafilter $U$
on $\kappa$. By Theorem \ref{th:gaifman}, $U$ is fully iterable. Let
$j_\alpha:M_\alpha\to N_\alpha$ be the $\alpha^{\text{th}}$-iterated
ultrapower of $U$. Observe that
$\Power(\kappa)^M=\Power(\kappa)^{M_\alpha}$ for all $\alpha$. By
remark \ref{rem:model} (1), $M_\alpha\cap L$ contains $L_\alpha$.
Thus, for a large enough $\alpha$, we have
$\Power(\kappa)^L\subseteq L_\alpha\subseteq M_\alpha$. It follows
that $L_{\kappa^+}\subseteq M$. Thus, $j_\alpha$ restricts to an
embedding on $L_{\kappa^+}$ and hence $0^\#$ exists.
\end{proof}
We will first show that if $0^\#$ exists, then the Silver
indiscernibles are $\alpha$-iterable in $L$ for all
$\alpha<\omega_1^L$. Later, we will modify this argument to show
that for $\alpha<\omega_1^L$, the $\alpha$-iterable cardinals are
downward absolute to $L$. We begin with the case of 1-iterable
cardinals. We will make use of a standard lemma below (see
\cite{welch:codingtheuniverse} for a proof).
\begin{lemma}\label{le:indheight}
If $0^\#$ exists and $\kappa$ is a Silver indiscernible, then
$\tmop{cf}^V((\kappa^+)^L)=\omega$.
\end{lemma}

\begin{theorem}\label{th:ind}
If\/ $0^\#$ exists, then the Silver indiscernibles are $1$-iterable
in $L$.
\end{theorem}
\begin{proof}
Let $I=\{i_\xi\mid \xi\in Ord\}$ be the Silver indiscernibles
enumerated in increasing order. Fix $\kappa\in I$ and let
$\lambda=(\kappa^+)^L$. Define $j:I\to I$ by $j(i_\xi)=i_\xi$ for
all $i_\xi<\kappa$ and $j(i_\xi)=i_{\xi+1}$  for all
$i_\xi\geq\kappa$ in $I$. The map $j$ extends, via the Skolem
functions, to an elementary embedding $j:L\to L$ with critical point
$\kappa$. Restrict to $j:L_\lambda\to L_{j(\lambda)}$, which is
clearly $\kappa$-powerset preserving.
Let $U$ be the weakly amenable $L_\lambda$-ultrafilter generated by
$\kappa$ via $j$ as in Proposition \ref{prop:genult}. Since every
$\alpha<\lambda$ has size $\kappa$ in $L_\lambda$, by weak
amenability, $U\cap L_\alpha$ is an element of $L_\lambda$.
Construct, using Lemma \ref{le:indheight}, a sequence $\la
\lambda_i:i\in\omega\ra$ cofinal in $\lambda$, such that each
$L_{\lambda_i}\prec L_\lambda$, and $U\cap
L_{\lambda_i},L_{\lambda_i}\in L_{\lambda_{i+1}}$. Let $j_i$ be the
restriction of $j$ to $L_{\lambda_i}$. Each $j_i:L_{\lambda_i}\to
L_{j(\lambda_i)}$ is an element of $L$ by remark \ref{rem:model}
(2), since it has size $\kappa$ in $L_\lambda$. These observations
motivate the construction below.

To show that $\kappa$ is 1-iterable in $L$, for every
$A\subseteq\kappa$ in $L$, we need to construct in $L$ a weak
$\kappa$-model $M$ containing $A$ and a 1-good $M$-ultrafilter on
$\kappa$. Fix $A\subseteq\kappa$ in $L$. Define in $L$, the tree $T$
of finite sequences of the form
\begin{displaymath}
s=\la h_0:L_{\gamma_0}\to L_{\delta_0},\ldots,\break
h_n:L_{\gamma_n}\to L_{\delta_n}\ra
\end{displaymath}
ordered by extension and satisfying the properties:
\begin{itemize}
\item[(1)] $A\in L_{\gamma_0}\models\rm{ZFC}^-$,
\item[(2)] $h_i:L_{\gamma_i}\to L_{\delta_i}$ is an elementary
embedding with critical point $\kappa$,
\item[(3)] $\delta_i<j(\lambda)$.
\end{itemize}
Let $W_i$ be the $L_{\gamma_i}$-ultrafilter generated by $\kappa$
via $h_i$. Then:
\begin{itemize}
\item[(4)] for $i<j\leq n$, we have $L_{\gamma_i},W_i\in
L_{\gamma_j}$, $L_{\gamma_i}\prec L_{\gamma_j}$, $L_{\delta_i}\prec
L_{\delta_j}$, and $h_j$ extends $h_i$.
\end{itemize}
We view the sequences $s$ as better and better approximations to the
embedding we are trying to build.

Consider the sequences
\begin{displaymath}
s_n=\la j_0:L_{\lambda_0}\to
L_{j(\lambda_0)},\ldots,j_n:L_{\lambda_n}\to L_{j(\lambda_n)}\ra.
\end{displaymath}
Clearly each $s_n$ is an element of $T$ and \hbox{$\la
s_n:n\in\omega\ra$} is a branch through $T$ in $V$. Hence the tree
$T$ is ill-founded. By absoluteness, it follows that $T$ is
ill-founded in $L$ as well. Let $\{h_i:L_{\gamma_i}\to
L_{\delta_i}:i\in \omega\}$ be a branch of $T$ in $L$ and $W_i$ be
the $L_{\gamma_i}$-ultrafilters as above. Let
\begin{displaymath}
h=\cup_{i\in\omega} h_i,\, L_\gamma=\cup_{i\in\omega}
L_{\gamma_i},\, L_\delta=\cup_{i\in\omega} L_{\delta_i}, \text{ and
} W=\cup_{i\in\omega} W_i.
\end{displaymath}
It is clear that $h:L_\gamma\to L_\delta$ is an elementary embedding
with critical point $\kappa$ and $W$ is a weakly amenable
$L_\gamma$-ultrafilter generated by $\kappa$ via $h$. Since the
ultrapower of $L_\gamma$ by $W$ is a factor embedding of $h$, it
must be well-founded.

We have now found a weak $\kappa$-model $L_\gamma$ containing $A$
for which there exists a 1-good $L_\gamma$-ultrafilter on
$\kappa$. This completes the proof
that $\kappa$ is 1-iterable.
\end{proof}
The next lemma will allow us to modify the proof of Theorem
\ref{th:ind} to show that 1-iterable cardinals are downward absolute
to $L$.
\begin{lemma}\label{le:countheight}
If $\kappa$ is a $1$-iterable cardinal, then every
$A\subseteq\kappa$ is contained in a weak $\kappa$-model $M$ for
which there exists a $1$-iterable $M$-ultrafilter $U$ on $\kappa$
satisfying the conditions:
\begin{itemize}
\item[(1)] $M=\cup_{n\in\omega}M_n$,
\item[(2)] $M_n,M_n\cap U\in M_{n+1}$,
\item[(3)] for $i<j$, we have $M_i\prec M_j$,
\item[(4)] $M\models$ ``I am $\her{\kappa}$" \emph{(}every set has transitive closure of size at most $\kappa$\emph{)}.
\end{itemize}
\end{lemma}
\begin{proof}
Fix $A\subseteq\kappa$ and find a weak $\kappa$-model $M'$
containing $A$ with a 1-good $M'$-ultrafilter $U'$ on $\kappa$. Let
$h:M'\to N'$ be the ultrapower embedding by $U'$. We can assume
without loss of generality that $M'\models$ ``I am $\her{\kappa}$"
by taking $\{B\in M'\mid$ $M'\models$ transitive closure of $B$ has
size $\leq\kappa\}$ instead of $M'$ and restricting the embedding
accordingly. Since $M'\models$ ``I am $\her{\kappa}$" and $h$ is
$\kappa$-powerset preserving, it follows that
$M'=\her{\kappa}^{N'}$. In $N'$, let $M_0$ be a transitive
elementary submodel of $\her{\kappa}$ of size $\kappa$ containing
$A$. Since $U'$ is weakly amenable, it follows that $U_0=M_0\cap U'$
is an element of $\her{\kappa}^{N'}$. Let $M_1$ be a transitive
elementary submodel of $\her{\kappa}$ of size $\kappa$ containing
$M_0$ and $U_0$ and let $U_1=M_1\cap U'$. Again, $U_1$ is clearly an
element of $\her{\kappa}^{N'}$. Inductively let $M_{n+1}$ be a
transitive elementary submodel of $\her{\kappa}$ of size $\kappa$
containing $M_n$ and $U_n$ and $U_{n+1}=M_{n+1}\cap U'$. Let
$M=\cup_{n\in\omega} M_n$ and $U=\cup_{n\in\omega} U_n$. Clearly $U$
is a weakly amenable $M$-ultrafilter and the ultrapower of $M$ by
$U$ is well-founded as it embeds into $N$.
\end{proof}

\begin{theorem}\label{th:1iterabledown}
If $\kappa$ is $1$-iterable, then $\kappa$ is $1$-iterable in $L$.
\end{theorem}
\begin{proof}
Observe that if $0^\#$ exists, the theorem follows from Theorem
\ref{th:ind} since all uncountable cardinals of $V$ are among the
Silver indiscernibles. So suppose $0^\#$ does not exist. In $L$, fix
$L_\xi$ of size $\kappa$. Choose a weak $\kappa$-model $M$
containing $L_\xi$ and $V_\kappa$ for which there exists a
1-iterable $M$-ultrafilter $U$ and let $j:M\to N$ be the ultrapower
embedding. It is easy to see that $\kappa$ is weakly compact in $N$,
and hence in $V_{j(\kappa)}^N\models$ ZFC. Since $V_{j(\kappa)}^N$
knows that $0^\#$ does not exist, it must satisfy that
$(\kappa^+)^L=\kappa^+$. Restrict to $j:L_\alpha\to L_\beta$ where
$\alpha$ and $\beta$ are the heights of $M$ and $N$ respectively. By
the observation above, $(\kappa^+)^{L_\beta}=\alpha$ and hence the
restriction is $\kappa$-powerset preserving. Note also, that by
Lemma \ref{le:countheight}, we can assume that
$\tmop{cf}^V(\alpha)=\omega$. Therefore the embedding $j:L_\alpha\to
L_\beta$ has exactly the same properties as the embedding with the
Silver indiscernible as the critical point. So we can proceed as in
the proof of Theorem \ref{th:ind} to construct a weak $\kappa$-model
containing $L_\xi$ and a 1-iterable ultrafilter for it in $L$.
\end{proof}
The next lemma is a simple observation that will prove key to
generalizing the arguments above for $\alpha$-iterable cardinals.
Let us say that
\begin{displaymath}
 \{j_{\xi\gamma}:M_\xi\to
M_\gamma\mid\break\xi<\gamma<\alpha\}
\end{displaymath}
is a \emph{good commuting system of elementary embeddings of length
$\alpha$} if:
\begin{itemize}
\item[(1)] for all $\xi_0<\xi_1<\xi_2<\alpha$, $j_{\xi_1\xi_2}\circ
j_{\xi_0\xi_1}=j_{\xi_0\xi_2}$.
\end{itemize}
Let $\kappa_\xi$ be the critical point of $j_{\xi\xi+1}$ and let
$U_\xi$ be the $M_\xi$-ultrafilter generated by $\kappa_\xi$ via
$j_{\xi\xi+1}$, then:
\begin{itemize}
\item [(2)] for all $\xi<\beta<\alpha$, if $A\in M_\xi$ and $A\subseteq U_\xi$, then
$j_{\xi\beta}(A)\subseteq U_{\beta}$.
\end{itemize}
\begin{remark}
The directed system of embeddings resulting from the iterated
ultrapowers construction is a good commuting system of elementary
embeddings.
\end{remark}
The next lemma shows that existence of good commuting systems of
elementary embeddings of length $\alpha$ is basically equivalent to
existence of $\alpha$-good ultafilters.
\begin{lemma}\label{le:iterable}
Suppose $\{j_{\xi\gamma}:M_\xi\to M_\gamma:\xi<\gamma<\alpha\}$ is a
good commuting system of elementary embeddings of length $\alpha$.
Suppose further that $m_0=\cup_{i\in\omega}m^{(i)}_0$ is a
transitive elementary submodel of $M_0$ such that $m_0^{(i)}\prec
m_0^{(i+1)}$ and $U_0\cap m_0^{(i)}$, $m_0^{(i)}\in m_0^{(i+1)}$.
Then $u_0=m_0\cap U_0$ is an $\alpha$-good $m_0$-ultrafilter.
\end{lemma}
\begin{proof}
Let $\{h_{\xi\gamma}:m_\xi\to m_\gamma:\xi<\gamma<\alpha\}$ be the
not necessarily well-founded directed system of embeddings obtained
by iterating $u_0$ and let $u_\xi$ be the $\xi^{\text{th}}$-iterate
of $u_0$. Define $u_0^{(i)}=u_0\cap m^{(i)}_0$ and
$u_\xi^{(i)}=h_{0\xi}(u^{(i)}_0)$. It is easy to see that
$h_{\beta\xi}(u^{(i)}_\beta)=u^{(i)}_\xi$ and
$u_\xi=\cup_{i\in\omega} u_\xi^{(i)}$. To show that each $m_\xi$ is
well-founded, we will argue that we can define elementary embeddings
$\pi_\xi:m_\xi\to M_\xi$. More specifically we will construct the
following commutative diagram:
\begin{diagram}
M_0  &\rTo^{j_{01}}& M_1&\rTo^{j_{12}}&M_2&\rTo^{j_{23}}  &\ldots&\rTo^{j_{\xi\xi+1}}&M_{\xi+1}&\rTo^{j_{\xi+1\xi+2}}&\ldots\\
\uTo_{\pi_0}& & \uTo_{\pi_1}& & \uTo_{\pi_2}& &\ldots& &\uTo_{\pi_{\xi+1}}& & &\\
m_0 & \rTo^{h_{01}} & m_1
&\rTo^{h_{12}}&m_2&\rTo{h_{23}}&\ldots&\rTo^{h_{\xi\xi+1}}&m_{\xi+1}&\rTo^{h_{\xi+1\xi+2}}&\ldots
\end{diagram}
where
\begin{itemize}
\item[(1)]
$\pi_{\xi+1}([f]_{u_\xi})=j_{\xi\xi+1}(\pi_\xi(f))(\kappa_\xi)$,
\item[(2)] if $\lambda$ is a limit ordinal and $t$ is a thread in the
direct limit $m_\lambda$ with domain $[\beta, \lambda)$, then
$\pi_\lambda(t)=j_{\beta\lambda}(\pi_\beta(t(\beta)))$,
\item[(3)] $\pi_\xi(u_\xi^{(i)})\subseteq U_\xi$ for all $i\in\omega$.
\end{itemize}
We will argue that the $\pi_\xi$ exist by induction on $\xi$. Let
$\pi_0$ be the identity map. Suppose inductively that $\pi_\xi$ has
the desired properties. Define $\pi_{\xi+1}$ as in (1) above. Since
$\pi_\xi(u_\xi^{(i)})\subseteq U_\xi$ by the inductive assumption,
it follows that $\pi_{\xi+1}$ is a well-defined elementary
embedding. The commutativity of the diagram is also clear. It
remains to verify that $\pi_{\xi+1}(u_{\xi+1}^{(i)})\subseteq
U_{\xi+1}$. Recall that
\begin{displaymath}
u_{\xi+1}^{(i)}=h_{\xi\xi+1}(u_{\xi}^{(i)})=[c_{u_\xi^{(i)}}]_{u_\xi}.
\end{displaymath}
Let $\pi_\xi(u_\xi^{(i)})=v$. Then by inductive assumption,
$v\subseteq U_\xi$. Thus,
\begin{displaymath}
\pi_{\xi+1}(u_{\xi+1}^{(i)})=j_{\xi\xi+1}(c_v)(\kappa_\xi)=c_{j_{\xi\xi+1}(v)}(\kappa_\xi)=j_{\xi\xi+1}(v).
\end{displaymath}
By hypothesis, $j_{\xi\xi+1}(v)\subseteq U_{\xi+1}$. This completes
the inductive step. The limit case also follows easily.
\end{proof}
Below, we give another useful example of a good commuting system of
elementary embeddings.
\begin{lemma}\label{le:silverult}
Suppose $L$ \emph{(}or $L_\rho$\emph{)} is the Skolem closure of a
collection of order indiscernibles $I$ and let $i_\xi$ be the
$\xi^{\text{th}}$ element of $I$. Suppose $\delta$ is below
o.t.\pl$I$\pr\ and $\lambda$ is an ordinal such that for all $\xi$
below o.t.\pl$I$\pr\ and all $\xi'<\lambda$, the sum $\xi+\xi'$ is
below o.t.\pl$I$\pr. Then the system of embeddings
$\{j_{\alpha\beta}\mid \alpha<\beta<\lambda\}$ defined by
$j_{\alpha\beta}(i_\xi)=i_\xi$ for $\xi<\delta+\alpha$ and otherwise
$j_{\alpha\beta}(i_\xi)=i_{\xi+\gamma}$ where $\alpha+\gamma=\beta$
 is a good commuting system of elementary embeddings.
\end{lemma}
Note that o.t.($I$) is allowed to be $Ord$. The proof is a
straightforward application of indiscernibility. Thus, if $\kappa$
is a Silver indiscernible, then there is a good commuting system of
elementary embeddings of length $Ord$ with the first embedding
having critical point $\kappa$.

Now we can generalize Lemma \ref{le:countheight} to the case of
$\alpha$-iterable cardinals.
\begin{lemma}\label{le:countheight2}
If $\kappa$ is an $\alpha$-iterable cardinal, then every
$A\subseteq\kappa$ is contained in a weak $\kappa$-model $M$ for
which there exists an $\alpha$-good $M$-ultrafilter $U$ on $\kappa$
satisfying conditions \emph{(1)-(4)} of Lemma
\emph{\ref{le:countheight}}.
\end{lemma}
\begin{proof}
Start with any weak $\kappa$-model $M'$ and an $\alpha$-good
$M'$-ultrafilter $U'$ on $\kappa$. Use proof of Lemma
\ref{le:countheight} to find a transitive elementary submodel $M$ of
$M'$ satisfying the requirements and use Lemma \ref{le:iterable} to
argue that $U=M\cap U'$ is $\alpha$-good.
\end{proof}
We are now ready to show that if $0^{\#}$ exists, the Silver
indiscernibles are $\alpha$-iterable in $L$ for all
$\alpha<\omega_1^L$. It will follow using the same techniques that
for $\alpha<\omega_1^L$, the $\alpha$-iterable cardinals are
downward absolute to $L$.
\begin{theorem}\label{th:silveriterable}
If $0^{\#}$ exists, then the Silver
indiscernibles are $\alpha$-iterable in $L$ for all $\alpha<\omega_1^L$.
\end{theorem}
\begin{proof}
Fix a Silver indiscernible $\kappa_0$ and let
$\lambda_0=(\kappa_0^+)^L$. Let $U_0$ be an $\omega_1$-good
$L_{\lambda_0}$-ultrafilter on $\kappa_0$ that exists by Lemma
\ref{le:iterable} combined with Lemma \ref{le:silverult}. Let
\begin{displaymath}
\{j_{\gamma\xi}:L_{\lambda_\gamma}\to
L_{\lambda_\xi}:\gamma<\xi<\alpha\}
\end{displaymath}
be the good commuting system of elementary embeddings obtained from
the first $\alpha$-steps of the iteration. Let $\kappa_\xi$ be the
critical point of $j_{\xi\xi+1}$ and let $U_\xi$ be the
$\xi^{\text{th}}$-iterate of $U_0$. As before, we find a cofinal
sequence $\la\lambda_0^{(i)}:i\in\omega\ra$ in $\lambda_0$ such that
$L_{\lambda_0^{(i)}}\prec L_{\lambda_0}$ and $U_0\cap
L_{\lambda_0^{(i)}}, L_{\lambda_0^{(i)}}\in L_{\lambda_0^{(i+1)}}$.
Let $U_0^{(i)}=U_0\cap L_{\lambda_0^{(i)}}$,
$L_{\lambda_\xi^{(i)}}=j_{0\xi}(L_{\lambda_0^{(i)}})$, and
$U_\xi^{(i)}=j_{0\xi}(U_0^{(i)})$. Finally, let
$j_{\gamma\xi}^{(i)}$ be the restriction of $j_{\gamma\xi}$ to
$L_{\lambda_\gamma^{(i)}}$. Observe that each $j_{\gamma\xi}^{(i)}$
is an element of $L$ by remark \ref{rem:model} (2). As before, we
will use these sequences to show that the tree we construct below is
ill-founded.

To show that $\kappa_0$ is $\alpha$-iterable in $L$, for every
$A\subseteq\kappa_0$ in $L$, we need to construct in $L$ a weak
$\kappa_0$-model $M$ containing $A$ and an $\alpha$-good
$M$-ultrafilter on $\kappa_0$. Fix $A\subseteq\kappa_0$ in $L$.
Also, fix in $L$, a bijection $\rho:[\alpha]^2\to\omega$.

Define in $L$, the tree $T$ of finite tuples of sequences $t=\la
s_0,\ldots,s_{n}\ra$ where each $s_i$ is a sequence of length $n$
consisting of approximations to the elementary embedding from stage
$\xi$ to stage $\beta$ where $\rho(\xi,\beta)=i$. That is
\begin{displaymath}
s_i=\la h_{\xi\beta}^{(0)}:L_{\gamma_{\xi}^{(0)}}\to
L_{\gamma_{\beta}^{(0)}},\ldots,h_{\xi\beta}^{(n)}:L_{\gamma_{\xi}^{(n)}}\to
L_{\gamma_{\beta}^{(n)}}\ra\end{displaymath}
 Note that if
$t$ is an $m$-tuple, then we require all sequences in the tuple to
have length $m$. We define $t\leq t'$ whenever the length of $t'$ is
greater than or equal to the length of $t$ and the $i^{\text{th}}$
coordinate of $t'$ extends the $i^{\text{th}}$ coordinate of $t$.
The sequences $s_i=\la h_{\xi\beta}^{(j)}:L_{\gamma_{\xi}^{(j)}}\to
L_{\gamma_{\xi}^{(j)}}\mid 0\leq j\leq n\ra$ are required to satisfy
the properties:
\begin{itemize}
\item[(1)]
$A\in L_{\gamma_0^{(0)}}\models\rm{ZFC}^-$,
\item[(2)]
$h_{\xi\beta}^{(j)}:L_{\gamma_\xi^{(j)}}\to L_{\gamma_\beta^{(j)}}$
are commuting elementary embeddings,
\item[(3)]
$\gamma_\xi^{(j)}<\lambda_\alpha=\cup_{\xi<\alpha}\lambda_\xi$.
\end{itemize}
Let $\kappa_\xi$ be the critical points of $h_{\xi\xi+1}^{(j)}$ and
let $W_\xi^{(j)}$ be the $L_{\gamma_{\xi}^{(j)}}$-ultrafilters
generated by $\kappa_\xi$ via $h_{\xi\xi+1}^{(j)}$, then:
\begin{itemize}
\item[(4)] for $j<k\leq n$,
$L_{\gamma_\xi^{(j)}},W_\xi^{(j)}\in L_{\gamma_\xi^{(k)}}$,
$L_{\gamma_\xi^{(j)}}\prec L_{\gamma_\xi^{(k)}}$ and
$h_{\xi\beta}^{(k)}$ extends $h_{\xi\beta}^{(j)}$,
\item[(5)]
$h_{\xi\beta}^{(j)}(W_\xi^{(j)})=W_\beta^{(k)}$.
\end{itemize}
As before, we argue using the iterated ultrapowers of $U_0$, that
$T$ is ill-founded in $V$ and hence in $L$. Unioning up the branch
in $L$, we obtain a good commuting system of elementary embeddings
of length $\alpha$ and therefore an $\alpha$-good ultrafilter for
the first model in the system by Lemma \ref{le:iterable}.
\end{proof}

\begin{theorem}
For $\alpha<\omega_1^L$, if $\kappa$ is $\alpha$-iterable, then
$\kappa$ is $\alpha$-iterable in $L$.
\end{theorem}
\begin{proof}
Use the proof of Theorem \ref{th:1iterabledown} together with Lemma
\ref{le:iterable}.
\end{proof}
\section{The hierarchy of $\alpha$-iterable
cardinals}\label{sec:hierarchy} In this section, we show using the
techniques developed in the previous section, that for
$\alpha\leq\omega_1$, the $\alpha$-iterable cardinals form a
hierarchy of strength. We make some observations about the
relationship between $\alpha$-iterable cardinals and
$\alpha$-Erd\H{o}s cardinals. We show that a 2-iterable cardinal is
a limit of Schindler's \emph{remarkable} cardinals, improving the
upper bound on their consistency strength, and that a remarkable
cardinal implies the existence of a countable transitive model of
ZFC with a proper class of 1-iterable cardinals. Finally, we answer
a question from \cite{gitman:ramsey} about whether 1-iterable
cardinals imply the existence of $\kappa$-powerset preserving
embeddings on weak $\kappa$-models satisfying full ZFC.
\begin{theorem}
If $\kappa$ is an $\alpha$-iterable cardinal, then for $\xi<\alpha$,
the cardinal $\kappa$ is a limit of $\xi$-iterable cardinals.
\end{theorem}
\begin{proof}
Suppose $\kappa$ is an $\alpha$-iterable cardinal. Choose a weak
$\kappa$-model $M_0$ containing $V_\kappa$ as an element for which
there exists an $\alpha$-good $M_0$-ultrafilter, satisfying the
conclusions of Lemma \ref{le:countheight2}. Let $j_\xi:M_\xi\to
M_{\xi+1}$ be the $\xi^{\text{th}}$ step of the iteration by $U_0$.
Suppose, first, that $\alpha=\beta+1$ is a successor ordinal. In
this case, the iteration will have a final model namely $M_\alpha$.
It suffices to argue that $\kappa$ is $\beta$-iterable in
$M_\alpha$. To see this, suppose that $\kappa$ is $\beta$-iterable
in $M_\alpha$, then $\kappa$ is $\beta$-iterable in $M_1$ as well.
But then $M_1$ satisfies that there is a $\beta$-iterable cardinal
below $j_0(\kappa)$ and hence, by elementarity, $M_0$ satisfies that
$\kappa$ is a limit of $\beta$-iterable cardinals. But since
$V_\kappa\in M_0$, the model must be correct about this assertion.
Now we exactly follow the argument that Silver indiscernibles are
$\beta$-iterable in $L$, with $M_\alpha$ in the place of $L$. Let
$M_0=\cup_{i\in\omega} M_0^{(i)}$ where the $M_0^{(i)}$ satisfy the
conclusions of Lemma \ref{le:countheight2} and let
$M_\xi^{(i)}=j_{0\xi}(M_0^{(i)})$ for $\xi\leq\alpha$. Observe that
for $\xi<\gamma\leq\beta$, the restrictions
$j_{\xi\gamma}^{(i)}:M_\xi^{(i)}\to M_\gamma^{(i)}$ are all elements
of $M_\beta$ by remark \ref{rem:model} (2). Moreover, $M_\beta$ is a
set in $M_\alpha$ and hence the ordinals of $M_\beta$ are bounded in
$M_\alpha$ by some ordinal $\delta$. This suffices to run the same
tree building argument. The bound $\delta$ is needed to insure that
the tree is a set. Next, suppose, that $\alpha$ is a limit ordinal.
Fix $\xi<\alpha$, then $\kappa$ is $\xi+1$-iterable. So by the
inductive assumption, $\kappa$ is a limit of $\xi$-iterable
cardinals.
\end{proof}

Next, we give some results on the relationship between
$\alpha$-iterable cardinals and $\alpha$-Erd\H{o}s cardinals for
$\alpha\leq\omega_1$.
\begin{definition}
Suppose $\kappa$ is a regular cardinal and $\alpha$ is a limit
ordinal. Then $\kappa$ is $\alpha$-\emph{Erd\H{o}s} if every
structure of the form $\la L_\kappa[A],A\ra$ where
$A\subseteq\kappa$ has a good set of indiscernibles of order type
$\alpha$.\footnote{See Section \ref{sec:vram} for a discussion of
\emph{good} sets of indiscernibles.}
\end{definition}
Equivalently, $\kappa$ is $\alpha$-Erd\H{o}s if it is least such
that the partition relation \hbox{$\kappa\rightarrow
(\alpha)^{<\omega}$} holds.

In \cite{welch:ramsey}, Sharpe and Welch showed that:
\begin{theorem}\label{th:erdos}
An $\omega_1$-Erd\H{o}s is a limit of $\omega_1$-iterable cardinals.
\end{theorem}
Here we show that:
\begin{theorem}
If $\kappa$ is a $\gamma$-Erd\H{o}s cardinal for some
$\gamma<\omega_1$ and $\delta<\gamma$ is an ordinal such that for
all $\xi'<\gamma$ and $\xi<\delta$, the sum $\xi'+\xi<\gamma$, then
there is a countable ordinal $\alpha$ and a real $r$ such that
$L_\alpha[r]$ is a model of \emph{ZFC} having a proper class of
$\delta$-iterable cardinals.
\end{theorem}
\begin{proof}
Suppose $\kappa$ is a $\gamma$-Erd\H{o}s cardinal, then there is a
set $I=\{i_\xi\mid \xi\in\gamma\}$ of good indiscernibles for
$L_\kappa[r]$ where $r$ codes the fact that $\gamma$ is a countable
ordinal. Let $L_\alpha[r]$ be the collapse of the Skolem closure of
$I$ in $L_\kappa[r]$, then it is the Skolem closure of some
collection $K=\{k_\xi\mid \xi\in\gamma\}$ of indiscernibles. The
indiscernibles $k_\xi$ are unbounded in $\alpha$. Since each
$L_{i_\xi}[r]\prec L_\kappa[r]$, we have $L_\beta[r]\prec
L_\kappa[r]$ where $\beta$ is the sup of $I$. It follows that the
Skolem closure of $I$ is contained in $L_\beta[r]$ and hence the
$k_\xi$ are unbounded in $L_\alpha[r]$. By Lemma \ref{le:silverult},
for every indiscernible $k_\xi$, there is a good commuting system of
elementary embeddings of length $\delta$ on $L_\alpha[r]$ with the
first embedding having critical point $k_\xi$. Thus, by Lemma
\ref{le:iterable}, there is a $\delta$-good ultrafilter for every
$L_{(k_\xi^+)^{L_\alpha[r]}}$. Since $\alpha$ is countable, it is
clear that $(k_\xi^+)^{L_\alpha[r]}$ has countable cofinality. Now
we can use the argument in the proof of Theorem
\ref{th:silveriterable} to show that each $k_\xi$ is
$\delta$-iterable in $L_\alpha[r]$. Notice that we cannot use these
techniques to make the argument for $\delta=\gamma$ since the tree
of embedding approximations must be a set in $L_\alpha[r]$.
\end{proof}
In particular, note that an $\omega$-Erd\H{os} cardinal implies for
every $n\in\omega$, the consistency of the existence of a proper
class of $n$-iterable cardinals.
\begin{remark}
$\gamma$-Erd\H{o}s cardinals do not necessarily have any iterability
since the least such cardinal need not be weakly compact.
\end{remark}
In \cite{schindler:remarkable}, Schinder defined \emph{remarkable}
cardinals and showed that they are equiconsistent with the assumption
that $L(\mathbb R)$ cannot be modified by proper forcing. Schindler
showed that an $\omega$-Erd\H{o}s cardinal implies that there is a
countable model with a remarkable cardinal. We show that if $\kappa$
is 2-iterable, then $\kappa$ is a limit of remarkable cardinals. By
Theorem \ref{th:erdos}, this is an improved upper bound on the
consistency strength of these cardinals.
\begin{definition}\label{def:rem}
A cardinal $\kappa$ is \emph{remarkable} if for each regular
$\lambda>\kappa$, there exists a countable transitive $M$ and an
elementary embedding $e:M\to H_\lambda$ with $\kappa\in ran(e)$ and
also a countable transitive $N$ and an elementary embedding
$\theta:M\to N$ such that:
\begin{itemize}
\item [(1)] $cp(\theta)=e^{-1}(\kappa)$,
\item[(2)] $Ord^M$ is a regular cardinal in $N$,
\item[(3)] $M=H^N_{Ord^M}$,
\item[(4)] $\theta(e^{-1}(\kappa))>Ord^M$.
\end{itemize}
\end{definition}
We will need the following property of 2-iterable cardinals.
\begin{theorem}\label{th:diagram}
If $\kappa$ is a $2$-iterable cardinal, then every $A\subseteq\kappa$
is contained in a weak $\kappa$-model $M\models \rm{ZFC}$ for which there exists an embedding
$j:M\to N$ such that $M=V_{j(\kappa)}^N$ and $M\prec N$.
\end{theorem}
For proof, it suffices to observe that if $U$ is a 2-good ultrafilter
for a weak $\kappa$-model $M$, then we get the following commutative diagram:
\begin{diagram}
M & &\rTo^{j_U}& &N=M/U\\
 & \rdTo(4,4)^{j_{U^2}}& & &\\
\dTo^{j_U}& & & &\dTo_{h_U}\\
 & & & & \\
N=M/U & &\rTo_{j_{j_U(U)}}& & K=N/j_U(U)
\end{diagram}
where $j_U$ and $h_U$ are ultrapowers by $U$ and $j_{j_U(U)}$ is the
ultrapower by $j_U(U)$. If $V_\kappa\in M$, the restriction of $h_U$
to $V_{j(\kappa)}^N$ has all the required properties. See
\cite{gitman:ramsey} for details. The ultrafilter needs to be 2-good
to ensure that the bottom arrow embedding has a well-founded target.
\begin{theorem}
If $\kappa$ is $2$-iterable, then $\kappa$ is a limit of remarkable cardinals.
\end{theorem}
\begin{proof}
Suppose $\kappa$ is 2-iterable, then there is $j:M\to N$ as in
Theorem \ref{th:diagram} with $V_\kappa\in M$. It will suffice to
argue that $\kappa$ is remarkable in $M$, since it will be
remarkable in $N$ by elementarity, and hence a limit of remarkable
cardinals. In $M$, fix a regular cardinal $\lambda>\kappa$.
Continuing to work in $M$, find $X_0\prec H_\lambda$ of size
$\kappa$ such that $V_\kappa\cup\{\kappa\}\subseteq X_0$. By remark
\ref{rem:model} (2), $j\restrict X_0:X_0\to j(X_0)$ is an element of
$N$. In $N$, find $Y_0\prec j(X_0)$ of size $\kappa$ such that
$X_0\cup j``X_0\cup \{\lambda\}\subseteq Y_0$. Let $j_0:X_0\to Y_0$
such that $j_0(x)=j(x)$ for all $x\in X_0$, then $j_0$ is clearly
elementary and an element of $N$. Let $Z_0=Y_0\cap H_\lambda$
(clearly $H_\lambda^M=H_\lambda^N$), then $Z_0\in H_\lambda$. Back
in $M$, find $X_1\prec H_\lambda$ such that $Z_0\subseteq X_1$ and
in $N$, find $Y_1\prec j(X_1)$ of size $\kappa$ and containing
$X_1\cup j``X_1$. Let $j_1:X_1\to Y_1$ such that $j_1(x)=j(x)$ for
all $x\in X_1$, then as before $j_1$ is elementary and an element of
$N$. Proceed inductively to define the the sequence $\la j_n:X_n\to
Y_n\mid n\in \omega\ra$. The elements of the sequence are all in
$N$, but the sequence itself need not be. As in the previous proofs,
we will use a tree argument to find a sequence with similar
properties in $N$ itself. The elements of the tree $T$ will be
sequences $\la h_0:P_0\to R_0,\ldots,h_n:P_n\to R_n\ra$ ordered by
extension and satisfying the properties:
\begin{itemize}
\item[(1)] $h_i:P_i\to R_i$ is an elementary
embedding with critical point $\kappa$,
\item[(2)] $V_\kappa\cup\{\kappa\}\subseteq P_0$, $P_i\in H_\lambda$,
$P_i$ has size $\kappa$ in $H_\lambda$,  $P_i\prec H_\lambda$, and
$P_i\subseteq R_i$
\item[(3)] $R_i\in H_{j(\lambda)}$ and $R_i$ has size $\kappa$,
\item[(4)] for $i<j\leq n$, $R_i\cap H_\lambda\subseteq P_j$.
\end{itemize}
The sequence of embeddings we constructed above is a branch through
$T$ and hence $T$ is ill-founded. Thus, $N$ has a branch of $T$. Let
$h:P\to R$ be the embedding obtained from unioning up the branch. By
our construction, $P\prec H_\lambda$, $P=H_\lambda^R$, and $P$ is an
element of $M$. Collapse $R$ and use the collapse to define an
elementary embedding of transitive structures
$\bar{h}:\bar{P}\to\bar{R}$ where $\bar{R}$ is the collapse of $R$
and $\bar{P}$ is the collapse of $P$. Since $\bar{h}$, $\bar{P}$,
and $\bar{R}$ all have transitive size $\kappa$, they are elements
of $M$. Observe that $Ord^{\bar{P}}=\gamma$ is a regular cardinal in
$\bar{R}$, the critical point of $\bar{h}$ is $\kappa$, and
$\bar{h}(\kappa)>\gamma$. Finally, in $M$,  take a countable
elementary substructure of $\la \bar{R},\bar{P},\bar{h}\ra$ and
collapse the structures to obtain an elementary embedding $i:m\to n$
of countable structures with critical point $\theta$. Let $e:m\to
H_\lambda$ be the composition of the inverses of the collapse maps.
The embeddings $i:m\to n$ and $e:m\to H_\lambda$ clearly satisfy
properties (1)-(4) in the definition of remarkable cardinals. This
completes the argument that $\kappa$ is remarkable in $M$.
\end{proof}
If $\kappa$ is at least 2-iterable, then by Theorem
\ref{th:diagram}, we can assume without loss of generality that we
have embeddings on weak $\kappa$-models satisfying full ZFC. Gitman
asked in \cite{gitman:ramsey} whether the same holds true for
1-iterable cardinals. We end this section, by answering the question
in the negative and using the same techniques to pin the consistency
strength of remarkable cardinals exactly between 1-iterable and
2-iterable cardinals.
\begin{theorem}\label{th:zfc}
If every $A\subseteq\kappa$ can be put into a weak $\kappa$-model
$M\models``\Power(\kappa)$ exists" for which there exists a $1$-good
$M$-ultrafilter on $\kappa$, then $\kappa$ is a limit of
$1$-iterable cardinals.
\end{theorem}
\begin{proof}
Fix a weak $\kappa$-model $M\models``\Power(\kappa)$ exists"
containing $V_\kappa$ for which there is a $1$-good $M$-ultrafilter
$U$, and let $j:M\to N$ be the ultrapower embedding. Fix
$A\subseteq\kappa$ in $N$ and find in $N$, a transitive
$M_0\prec\her{\kappa}$ of size $\kappa$ and containing $A$. As
before, $j\restrict M_0:M_0\to j(M_0)$ is in $N$. Next, find a
transitive $M_1\prec \her{\kappa}$ of size $\kappa$ containing $M_0$
and $U\cap M_0$ and proceed inductively to define the sequence $\la
M_n\mid n\in\omega\ra$ in this manner. Again, we construct a tree to
obtain a sequence with similar properties in $N$ that will witness
1-iterability. That the tree can be defined in the first place is a
consequence of the fact that $\Power(j(\kappa))$ and hence
$\her{j(\kappa)}$ exists in $N$ by elementarity.
\end{proof}
\begin{corollary}
If every $A\subseteq\kappa$ can be put into a weak $\kappa$-model
$M\models\rm{ZFC}$ for which there exists a $1$-good $M$-ultrafilter
on $\kappa$, then $\kappa$ is a limit of $1$-iterable cardinals.
\end{corollary}
\begin{theorem}
If $\kappa$ is a remarkable cardinal, then there is a countable
transitive model of $\rm{ZFC}$ with a proper class of $1$-iterable
cardinals.
\end{theorem}
\begin{proof}
Fix a regular $\lambda>\kappa^+$ and let $e:M\to H_\lambda$,
$\sigma:M\to N$ with critical point $e^{-1}(\kappa)=\delta$ be as in
definition \ref{def:rem}. Note that $\Power(\delta)$ exists in $M$
since $\Power(\kappa)$ exists in $H_\lambda$. Arguing exactly as in
the proof of Theorem \ref{th:zfc}, we see that
$\delta=e^{-1}(\kappa)$ is a limit of 1-iterable cardinals. Thus,
$V_\delta^M$ is a countable transitive model of ZFC with a class of
1-iterable cardinals.
\end{proof}

\section{Ramsey-like cardinals and downward absoluteness
to $K$}\label{sec:K} In this section, we show that the strongly
Ramsey and super Ramsey cardinals introduced in \cite{gitman:ramsey}
are downward absolute to the core model $K$.
\begin{definition}
A cardinal $\kappa$ is \emph{strongly Ramsey} if every
$A\subseteq\kappa$ is contained in a $\kappa$-model $M$ for which
there exists a $\kappa$-powerset preserving elementary embedding
$j:M\to N$.
\end{definition}
\begin{definition}
A cardinal $\kappa$ is \emph{super Ramsey} if every
$A\subseteq\kappa$ is contained in a $\kappa$-model
$M\prec\her{\kappa}$ for which there exists a $\kappa$-powerset
preserving elementary embedding $j:M\to N$.
\end{definition}
Strongly Ramsey cardinals are limits of \emph{completely Ramsey}
cardinals that top Feng's $\Pi_\alpha$-Ramsey hierarchy
\cite{feng:ramsey}. They are Ramsey, but not necessarily completely
Ramsey. They were introduced with the motivation of using them for
indestructibility arguments involving Ramsey cardinals. Such an
application is made in Section \ref{sec:vram}. Super Ramsey
cardinals are limits of strongly Ramsey cardinals and have the
advantage that the embedding is on a $\kappa$-model  that is
stationarily correct. Note that we can restate the definition of
strongly Ramsey and super Ramsey cardinals in terms of the existence
of weakly amenable $M$-ultrafilters. Since we require the embedding
to be on a $\kappa$-model, such an ultrafilter is automatically
countably complete and therefore has a well-founded ultrapower.

As a representative \emph{core model} $K$ here we take that
constructed using extender sequences which are
\emph{non-overlapping} (see \cite{zeman:inner}). \ In such a model a
strong cardinal may exist but not a sharp for such. The argument
does not depend on any particular fine structural considerations,
simply the definability of $K$ up to $\kappa^+$ in any
$H_{\kappa^+}$ with applications of the Weak Covering Lemma
(\tmtextit{cf. }{\hspace{0.25em}} {\cite{zeman:inner}}).

\begin{proposition}
  If $\kappa$ is strongly Ramsey, then $\kappa$ is strongly Ramsey in K.
\end{proposition}
\begin{proof}
Let $\kappa$ be strongly Ramsey and fix $A\subseteq\kappa$ in $K$.
Choose a $\kappa$-model $M$ containing $A$ such that $M\models A\in
K$ for which there exists a weakly amenable $M$-ultrafilter $U$ on
$\kappa$. To see that we can choose such $M$, note that $A \in P =
\langle J_{\alpha}^{E^K}, \in, E^K \rangle_{}$ for some $\alpha <
\kappa^+$ where $E^K$ is the extender sequence from which $K$ is
constructed. We may assume that a code for $P$ is definable over $V$
as a subset of $\kappa$. Hence we may assume that the $M$ witnessing
strong Ramseyness has this code, and so $P$, as an element. Note
that with $P \in M$, $K_{\kappa}$ is an initial segment of $K^M$.
Finally observe that $V_\kappa\in M$ since $M$ is a $\kappa$-model.
Let $\bar{\kappa} = \text{($\kappa^+)^{K_M}$}$ and set $\bar{K}
=K_{\bar{\kappa}}^M$. The possibility that $\bar{\kappa} = Ord^M$ is
allowed.

Note that $P\in\bar{K}$ and moreover a standard comparison argument
shows that $\bar{K}$ is an initial segment of $K$. Consider the
structure $N=\la \bar{K}, \in,W\ra$ where $W=U\cap\bar{K}$, and
observe that $W$ is a weakly amenable $\bar{K}$-ultrafilter. Note
that cf($\bar{\kappa})=\kappa$. If $\bar{\kappa}=(\kappa^+)^M$, this
follows since $M$ is a $\kappa$-model. Otherwise, consider the inner
model $W_M=\cup_{\alpha\in Ord}H^{M_\alpha}_{\kappa_\alpha}$
obtained by iterating the ultrafilter $U$ out through the ordinals,
in which $\bar{\kappa}$ remains the $K$-successor of $\kappa$, and
apply the Weak Covering Lemma to $\bar{\kappa}$. Hence $N$ is a
premouse iterable by the ultrafilter $W$. This allows us to
coiterate $N$ with $K$. We note that for no $\mu<\kappa$ do we have
$o^K(\mu)\geq\kappa$, , that is $\kappa$ is not overlapped by any
extender on a critical point $\mu$ below $\kappa$ since otherwise
the ultrafilter $W$ would generate the sharp for an inner model with
a strong cardinal and we are only considering $K$ build using
non-overlapping extenders. Hence if $K$ were to move in this
coiteration, either $\bar{\kappa}=(\kappa^+)^K$ and $\kappa$ is
measurable in $K$ (and hence already strongly Ramsey) or else $K$ is
first truncated to some $N'\in K$, $N'=\la \bar{K},\in, F\ra$ with a
weakly amenable $\bar{K}$-ultrafilter $F$. The next paragraph shows
that $N'$ witnesses the strong Ramsey property for $A$.

It remains to show that $\bar{K}^{<\kappa}\subseteq\bar{K}$ in $K$.
Fix $\alpha<\kappa$ and $f : \alpha \to \bar{K}$ in $K$. Without
loss of generality we shall assume that $f$ is 1-1. Since $M$ is a
$\kappa$-model, we have $f\in M$. Suppose $f (\gamma) \in
\bar{K_{}}_{\delta (\gamma)}$ for some $\delta (\gamma) <
\bar{\kappa}$. Let $\delta = \sup_{\gamma < \alpha} \delta
(\gamma)$. Then $\delta < \bar{\kappa}$. Let $G \in \bar{K}$ be such
that $G : \kappa \to\bar{K_{}}_{\delta}$ is a bijection. Then if $C
=G^{- 1}$``$\tmop{ran} (f)$, then $C$ is a bounded subset of
$\kappa$ with $C \in K$. However $C \in V_{\kappa} \rightarrow C \in
K^M$. As both $f, G$ are (1-1) there is a permutation $\pi : \alpha
\longrightarrow C$ with $f = G \circ \pi$. Again $\pi \in V_{\kappa}
\cap K \cap M$. Hence $f \in \bar{K}$.
\end{proof}
\begin{proposition}
If $\kappa$ is super Ramsey, then $\kappa$ is super Ramsey in $K$.
\end{proposition}
\begin{proof}
Note that if $M \prec H_{\kappa^+}$, then $K^M \prec
(K)^{H_{\kappa^+}} = K_{\kappa^+}$. Let $\bar{\kappa}=(\kappa^+)^K$.
Then also $K^M\cap K_{\bar{\kappa}}\prec K_{\bar{\kappa}}$. Let
$\bar{K}=K^M\cap K_{\bar{\kappa}}$. Now argue as in the last
proposition using $N=\la\bar{K},\in,U\cap \bar{K}\ra$ where $U$ is
the filter weakly amenable to $M$.
\end{proof}
\section{Virtually Ramsey cardinals}\label{sec:vram}
In \cite{welch:ramsey}, Sharpe and Welch defined a new large
cardinal notion, the \emph{virtually Ramsey} cardinal. Virtually
Ramsey cardinals are defined by an apparently weaker statement about
the existence of good indiscernibles than Ramsey cardinals. The
definition was motivated by the conditions needed to get an upper
bound on the consistency strength of the Intermediate Chang's
Conjecture. In this section, we separate the notions of Ramsey and
virtually Ramsey cardinals using an old forcing argument of Kunen's
showing how to destroy and then resurrect a weakly compact cardinal
\cite{kunen:saturatedideals}.
\begin{definition}\label{def:goodind}
Suppose $\kappa$ is a cardinal and $A\subseteq\kappa$. Then
$I\subseteq\kappa$ is a \emph{good set of indiscernibles} for
$\lset{\kappa}$ if for all $\gamma\in I$:
\begin{itemize}
\item[(1)] $\la L_\gamma[A\cap \gamma],A\cap \gamma\ra\prec \lset{\kappa}$.
\item[(2)] $I\setminus\gamma$ is a set of indiscernibles for
$\lsetex{\kappa}_{\xi\in\gamma}$.
\end{itemize}
\end{definition}
\begin{remark}\label{rem:goodind}
If for every $A\subseteq\kappa$, there is $\gamma<\kappa$ such $\la
L_\gamma[A\cap\gamma], A\cap\gamma\ra\prec \la L_\kappa[A],A\ra$,
then it is easy to see that $\kappa$ must be inaccessible. Also, if
$I$ is a set of good indiscernibles for $\la L_\kappa[A],A\ra$ and
$|I|\geq 3$, then by clause (2), every $\gamma\in I$ is inaccessible
in $\la L_\kappa[A],A\ra$.
\end{remark}
\begin{theorem}
A cardinal $\kappa$ is Ramsey if and only if for every
$A\subseteq\kappa$, the structure $\lset{\kappa}$ has a good set of
indiscernibles of size $\kappa$.
\end{theorem}
See \cite{donder:goodinds} for details on good sets of
indiscernibles and proof of above theorem. In general, for
$A\subseteq\kappa$, let $\mathscr I_A=\{\alpha\in\kappa\mid$ there
is an unbounded set of good indiscernibles $I_\alpha\subseteq\alpha$
 for $\lset{\kappa}\}$.
\begin{definition}
A cardinal $\kappa$ is \emph{virtually Ramsey} if for every
$A\subseteq\kappa$, the set $\I_A$ contains a club.\footnote{The
original definition in \cite{welch:ramsey} required that $\I_A$
contain only an $\omega_1$-club.}
\end{definition}
There is no obvious reason to suppose that the good sets of
indiscernibles below each of the ordinals in $\I_A$ can be glued
together into a good set of indiscernibles of size $\kappa$,
suggesting that virtually Ramsey cardinals are not necessarily
Ramsey. First, we make some easy observations about virtually Ramsey
cardinals.
\begin{proposition}
Ramsey cardinals are virtually Ramsey.
\end{proposition}
\begin{proof}
Suppose $\kappa$ is a Ramsey cardinal. If $A\subseteq\kappa$, then
there is a good set indiscernibles $I$ of size $\kappa$ for the
structure $\lset{\kappa}$. Clearly the club of all limit points of
$I$ is contained in $\I_A$. This verifies that $\kappa$ is virtually
Ramsey.
\end{proof}
The next proposition confirms that being a virtually Ramsey cardinal
is a large cardinal notion.
\begin{proposition}
Virtually Ramsey cardinals are Mahlo.
\end{proposition}
\begin{proof}
Suppose $\kappa$ is virtually Ramsey. By remark \ref{rem:goodind},
$\kappa$ is inaccessible. To see that $\kappa$ is Mahlo, let
$A\subseteq\kappa$ code $H_\kappa$ and $C\subseteq\kappa$ be a club.
If $I$ is any good set of indiscernibles for $L_\kappa[A,C]$ and
$\gamma\in I$, then $\gamma\in C$ by (1) of \ref{def:goodind}. By
remark \ref{rem:goodind}, $L_\kappa[A,C]$ thinks that $\gamma$ is
inaccessible but it is correct about this since it contains all of
$H_\kappa$.
\end{proof}
Next, we give a sufficient condition needed to glue the good sets of
indiscernibles below the ordinals in $\I_A$ into a good set of
indiscernibles of size $\kappa$.
\begin{proposition}
If a cardinal is virtually Ramsey and weakly compact, then it is
Ramsey.
\end{proposition}
\begin{proof}
Suppose $\kappa$ is virtually Ramsey and weakly compact. We will
argue that we can glue together the good sets of indiscernibles
coming from the different ordinals of the club contained in $\I_A$
into a good set of indiscernibles of size $\kappa$. Fix
$A\subseteq\kappa$ and let $C$ be a club contained in $\I_A$. Fix
any weak $\kappa$-model $M$ containing $A$, $C$ and $V_\kappa$ as
elements. By weak compactness, there exists an embedding $j:M\to N$
with critical point $\kappa$. Observe that $M\models C\subseteq
\I_A^M$, where $\I_A^M$ is the set $\I_A$ defined from the
perspective of $M$. By elementarity $N\models j(C)\subseteq
\I_{j(A)}^N$. Since $\kappa\in j(C)$, it follows that there is a
good set of indiscernibles for $\la L_{j(\kappa)}[j(A)],j(A)\ra$
below $\kappa$. But since $j(A)\cap \kappa=A$, it is easy to see
that this is a good set of indiscernibles for $\lset{\kappa}$ as
well. This completes the proof that $\kappa$ is Ramsey.
\end{proof}

Our strategy to separate virtually Ramsey and Ramsey cardinals will
be to start with a Ramsey cardinal and force to destroy its weak
compactness while preserving virtual Ramseyness. Although ideally we
would like to start with a Ramsey cardinal, we will have to start
with a strongly Ramsey cardinal instead. The reason being that
strongly Ramsey cardinals have embeddings on sets with
$<\kappa$-closure that is required for indestructibility techniques.
The argument below was worked out jointly with Joel David Hamkins
and we would like to thank him for his contribution.

The forcing we use is Kunen's well-known forcing from
\cite{kunen:saturatedideals} to destroy and then resurrect weak
compactness. The next lemma is a key observation in the argument.
\begin{lemma}\label{lem:vrindes}
If $\p$ is a $<\kappa$-distributive, stationary preserving forcing,
$G\subseteq\p$ is $V$-generic and $\kappa$ is virtually Ramsey in
$V[G]$, then $\kappa$ was already virtually Ramsey in $V$.
\end{lemma}
\begin{proof}
Fix $A\subseteq\kappa$. Since $\p$ is $<\kappa$-distributive, it
cannot add any new good sets of indiscernibles to ordinals
$\alpha<\kappa$. It follows that $\I_A=\I_A^{V[G]}$. If $\I_A$ does
not contain a club in $V$, then the complement $\overline {\I}_A$ is
stationary in $V$. Since $\p$ is stationary preserving, $\overline
{\I}_A$ remains stationary in $V[G]$. This is clearly a
contradiction since $\kappa$ is virtually Ramsey in $V[G]$ and hence
$\I_A$ contains a club.
\end{proof}

First, we define a forcing $\q$ to add a Souslin tree $T$ together
with a group of automorphisms $\mathscr G$ that acts
\emph{transitively} on $T$. A group of automorphisms $\mathscr G$ of
a tree $T$ is said to act transitively if for every $a$ and $b$ on
the same level of $T$, there is $\pi\in \mathscr G$ with $\pi(a)=b$.
The elements of $\q$ will be pairs $(t,f)$ where $t$ is a
\emph{normal} $\alpha+1$-tree for some $\alpha<\kappa$ such that
$Aut(t)$ acts transitively and
$f:\lambda\underset{\text{onto}}{\overset{1-1}{\to}}Aut(t)$ is some
enumeration of $Aut(t)$. We have $(t_1,f_1)\leq (t_0,f_0)$ when
\begin{itemize}
\item[(1)] $t_1$ end-extends $t_0$,
\item[(2)] for all $\xi\in Dom(f_0)$, $f_1(\xi)$ extends $f_0(\xi)$.
\end{itemize}
The strategy will be to force with $\q$ to add a Souslin tree $T$
thereby destroying the strong Ramseyness of $\kappa$ and then to
force with $T$ itself to resurrect it. The argument that the second
forcing resurrects the strong Ramseyness of $\kappa$ will rely on
the fact that the combined forcing $\q$ followed by $T$ has a dense
subset that is $<\kappa$-closed. It is to obtain this result that
the usual forcing to add a Souslin tree needs to be augmented with
the automorphism groups.

To show that the generic $\kappa$-tree $T$ added by $\q$ is Souslin,
we need to argue that every maximal antichain of $T$ is bounded. In
the usual forcing to add a Souslin tree, the conditions are normal
$\alpha+1$-trees and the argument is made by proving the Sealing
Lemma. The Sealing Lemma states that if a condition forces that
$\dot{A}$ is a name for a maximal antichain, then there is a
stronger condition forcing that it is bounded. The argument for the
Sealing Lemma goes as follows:\\
Suppose $t_0\forces \dot{A}$ is a maximal antichain of $\dot{T}$.
Choose $t_1\leq t_0$ such that for every $s\in t_0$, there is
$a_s\in t_1$ compatible with $s$ and $t_1\forces a_s\in \dot{A}$.
Build a sequence $\cdots\leq t_n\leq \cdots\leq t_1\leq t_0$ such
that for every $s\in t_n$, there is $a_s\in t_{n+1}$ compatible with
$s$ and $t_{n+1}\forces a_s\in \dot{A}$. Let $t$ be the union of
$t_n$ and build the top level of $t$ by adding a branch through
every pair $s$ and $a_s$. Since every new branch passes through an
element of $\dot{A}$, this \emph{seals} the antichain. We will carry
out a similar argument with the forcing $\q$, but it will be
complicated by the fact that whenever we add a node on top of a
branch $B$, we need to add nodes on top of branches $f(\xi)``B$.
While $B$ passes through an element of $\dot{A}$, there is no reason
why $f(\xi)`` B$ should. In fact, since the automorphism groups act
transitively, it will suffice to add a single carefully chosen
branch to the limit tree of the conditions and take the limit level
to be all the images of the branch under the automorphism group on
the second coordinate. Thus, we need to build our sequence of
conditions such that the limit of the trees on the sequence has a
branch all of whose images under the automorphism group go through
elements of the antichain.
\begin{lemma}[Sealing Lemma]
Suppose $p$ is a condition in $\q$, $\dot{T}$ is the canonical
$\q$-name for the generic $\kappa$-tree added by $\q$, and $p\forces
\dot{A}$ is a maximal antichain of $\dot{T}$. Then there is $q\leq
p$ forcing that $\dot{A}$ is bounded.
\end{lemma}
\begin{proof}
Fix $p\forces \dot{A}$ is a maximal antichain of $\dot{T}$.  Let
$p=(t_0,f_0)$ with $t_0$ of height $\alpha+1$ and
$f_0:\lambda_0\underset{\text{onto}}{\overset{1-1}{\to}}Aut(t_0)$.
Choose some $M\prec\her{\kappa}$ of size $<\kappa$ containing $\q$,
$p$, and $\dot{A}$ with the additional property that
$Ord^M\cap\kappa=\beta$ is an initial segment of $\kappa$. We will
work inside $M$ to build a condition $(t,f)$ with $t$ of height
$\beta+1$ strengthening $(t_0,f_0)$ and sealing $\dot{A}$. We will
need a bookkeeping function
$\varphi:\kappa\underset{\text{onto}}{\to}\kappa$ with the property
that every $\xi$ appears in the range cofinally often. Notice that
by elementarity, $M$ contains some such function $\varphi$. Working
\emph{entirely inside} $M$, we carry out the following construction
for $\kappa$ many steps. By going to a stronger condition, we can
assume without loss of generality that there is $a\in t_0$ such that
$(t_0,f_0)\forces a\in\dot{A}$. Let $B_0$ be any branch through $a$
in $t_0$. Let $a_0$ be the top node of $B_0$. The node $a_0$ begins
the branch we are trying to construct. Let $(t_1,f_1)$ be a
condition in $\q$ strengthening $(t_0, f_0)$ and having the property
that for every $s\in t_0$, there is $a_s\in t_1$ compatible with $s$
such that $(t_1,f_1)\forces a_s\in\dot{A}$. Consult the bookkeeping
function $\varphi(1)=\gamma$. This will determine how $a_0$ gets
extended. If $\gamma\geq \lambda_0$, let $a_1$ be the node on the
top level of $t_1$ extending $a_0$. Otherwise, consider
$f_0(\gamma)$ and $f_0(\gamma)(a_0)=s$. Let $s'$ be on the top level
of $t_1$ above $s$ and $a_s$. Finally, let
$f_1(\gamma)^{-1}(s')=a_1$. This has the effect that no matter how
we extend $f_0(\gamma)$, the image under it of the branch we are
building will pass through $\dot{A}$. At successor stages
$\sigma+1$, we will extend the condition $(t_\sigma,f_\sigma)$ to a
condition $(t_{\sigma+1}, f_{\sigma+1})$ having the property that
for every $s\in t_\sigma$, there is $a_s\in t_{\sigma+1}$ compatible
with $s$ such that $(t_{\sigma+1},f_{\sigma+1})\forces
a_s\in\dot{A}$. Next, we will consult the bookkeeping function
$\varphi(\sigma)=\gamma$ and let it decide as above how $a_\sigma$
gets chosen. At limit stages $\lambda$, we will let $t_\lambda$ be
the union of $t_\xi$ for $\xi<\lambda$ and $f_\lambda$ be the
coordinate-wise union of $f_\xi$. Now use the branch through $a_\xi$
to define a limit level for $t_\lambda$, thereby extending to
$(t_{\lambda+1},f_{\lambda+1})$. From the perspective of $M$, we are
carrying out this construction for $\kappa$ many steps, but really
we are only carrying it out for $\beta$ many steps. In $V$, we build
$(t,f)$ by unioning the sequence and adding a limit level using the
branch of the $a_\xi$. It should be clear that $(t,f)$ forces that
$\dot{A}$ is bounded.
\end{proof}
\begin{corollary}
The generic $\kappa$-tree added by $\q$ is Souslin.
\end{corollary}
In the generic extension by $\q$, the Souslin tree $T$ it adds can
be viewed as a poset. Next, we will argue that forcing with
$\q*\dot{T}$ is forcing equivalent $Add(\kappa,1)$, where
$Add(\kappa,1)$ is the forcing to add a Cohen subset to $\kappa$.
Since every $<\kappa$-closed poset of size $\kappa$ is forcing
equivalent to $Add(\kappa,1)$, it suffices argue that $\q*\dot{T}$
has a dense subset that is $<\kappa$-closed.
\begin{lemma}
The forcing $\q*\dot{T}$ has a dense subset that is
$<\kappa$-closed.
\end{lemma}
\begin{proof}
Conditions in $\q*\dot{T}$ are triples $(t,f, \dot{a})$ where $t$ is
an $\alpha+1$-tree, $f$ is an enumeration of the automorphism group
of $t$, and $\dot{a}$ is a name for an element of $\dot{T}$. We will
argue that conditions of the form $(t,f,a)$ where $a$ is on the top
level of $t$ form a dense $<\kappa$-closed subset of $\q*\dot{T}$.
Start with any condition $(t_0,f_0,\dot{b}_0)$ and strengthen
$(t_0,f_0)$ to a condition $(t_1,f_1)$ deciding that $\dot{b}$ is
$b\in t_1$. Now we have
$(t_0,f_0,\dot{b})\geq(t_1,f_1,b)\geq(t_1,f_1,a)$ where $a$ is above
$b$ on the top level of $t_1$.  Thus the subset is dense. Suppose
$\gamma<\kappa$ and we have a descending $\gamma$-sequence
$(t_0,f_0,a_0)\geq
(t_1,f_1,a_1)\geq\ldots\geq(t_\xi,f_\xi,a_\xi)\geq\ldots$. To find a
condition that is above the sequence, we take unions of the first
two coordinates and make the limit level of the tree in the first
coordinate consist of images of the branch through $\la a_\xi\mid
\xi<\gamma\ra$ under the automorphisms in the second coordinate.
\end{proof}
Let $\p_\kappa$ be the Easton support iteration which adds a Cohen
subset to every inaccessible cardinal below $\kappa$. We will force
with the iteration $\p_\kappa*\dot{\q}*\dot{T}$. This is equivalent
to forcing with $\p_\kappa*Add(\kappa,1)$. The forcing argument will
rely crucially on the following standard theorem about preservation
of strong Ramsey cardinals after forcing.
\begin{theorem}\label{th:strongramseyindes}
If $\kappa$ is strongly Ramsey in $V$, then it remains strongly
Ramsey after forcing with $\p_\kappa*Add(\kappa,1)$.
\end{theorem}
The proof uses standard techniques for lifting embeddings and will
appear in \cite{gitman:ramseyindes}.

 Now we have all the machinery
necessary to produce a model where $\kappa$ is virtually Ramsey but
not weakly compact.
 \begin{theorem}
If $\kappa$ is a strongly Ramsey cardinal, then there is a forcing
extension, in which $\kappa$ is virtually Ramsey, but not weakly
compact.
\end{theorem}
\begin{proof}
Let $\kappa$ be a strongly Ramsey cardinal and
$G*T*B\subseteq\p_\kappa*\dot{\q}*\dot{T}$ be $V$-generic. Since
$\q$ adds a Souslin tree, $\kappa$ is not weakly compact in the
intermediate extension $V[G][T]$. But by Theorem
\ref{th:strongramseyindes}, the strong Ramseyness of $\kappa$ is
resurrected in $V[G][T][B]$. Recall that forcing with a Souslin tree
is $<\kappa$-distributive. The Souslin tree forcing is also
$\kappa$-cc and hence stationary preserving. So by Lemma
\ref{lem:vrindes}, we conclude that $\kappa$ remains virtually
Ramsey in $V[G][T]$. Thus, in $V[G][T]$, the cardinal $\kappa$ is
virtually Ramsey but not weakly compact, and hence not Ramsey.
\end{proof}
 We showed starting from a strongly Ramsey cardinal that it is possible to have virtually Ramsey
cardinals that are not Ramsey, thus separating the two notions. The
following questions are still open.
\begin{question}
Can we separate Ramsey and virtually Ramsey cardinals starting with
just a Ramsey cardinal?
\end{question}
\begin{question}
Are virtually Ramsey cardinals strictly weaker than Ramsey
cardinals?
\end{question}
\begin{question}
Are virtually Ramsey cardinals downward absolute to $K$?
\end{question}
\bibliographystyle{alpha}
\bibliography{gitmanbib,logicbib}
\end{document}